\documentclass[10pt]{article}
\usepackage{a4}
\usepackage{latexsym}
\usepackage{amsfonts}
\usepackage{amssymb}
\usepackage{amsmath}
\usepackage{amsthm}

\newtheorem{theorem}{Theorem}
\newtheorem{lemma}[theorem]{Lemma}
\newtheorem{corollary}[theorem]{Corollary}
\newtheorem{proposition}[theorem]{Proposition}
\newtheorem{definition}[theorem]{Definition}

\newtheorem{claim}[theorem]{Claim}
\title{Explicit incidence bounds over general finite fields}
\author{Timothy G. F. Jones \footnote{Department of Mathematics, University of Bristol, BS8 1TW, United Kingdom, tgf.jones@bristol.ac.uk}}
\date{}

\begin{document}
\maketitle
\begin{abstract}
Let $\mathbb{F}_{q}$ be a finite field of order $q=p^k$ where $p$ is prime. Let $P$ and $L$ be sets of points and lines respectively in $\mathbb{F}_{q} \times \mathbb{F}_{q}$ with $|P|=|L|=n$. We establish the incidence bound $I(P,L) \leq \gamma n^{\frac{3}{2}-\frac{1}{12838}}$, where $\gamma$ is an absolute constant, so long as $P$ satisfies the conditions of being an `antifield'. We define this to mean that the projection of $P$ onto some coordinate axis has no more than half-dimensional interaction with large subfields of $\mathbb{F}_q$. In addition, we give examples of sets satisfying these conditions in the important cases $q=p^2$ and $q=p^4$.
\end{abstract}

\section*{Preliminary notation}
This paper uses the following notation throughout. Given two real-valued functions $f,g$ with domain $D$, we write 

\begin{itemize}
\item $f\ll g$, $f=O(g)$ or $g = \Omega(f)$ if there is a constant $\gamma$ such that $f(x) \leq \gamma g(x)$ for all $x \in D$. The implicit constant $\gamma$ may be different each time this notation is used.
\item $f \approx g$ if $f \ll g$ and $g \ll f$
\end{itemize}

Given two sets $A,B \subseteq \mathbb{F}_q$, we define:

\begin{itemize}
\item the \textbf{sumset} $A+B=\left\{a+b:a \in A, b \in B\right\}$
\item the \textbf{product set} $A \cdot B=\left\{ab:a \in A, b \in B\right\}$
\item the \textbf{ratio set} $\frac{A}{B}=\left\{a b^{-1}:a \in A, b \in B, b \neq 0\right\}$
\end{itemize}

\section{Introduction}
\subsection{Incidences}
This paper is about incidences between points and lines in a plane. A point is \textbf{incident} to a line if it lies on that line, and a single point can be incident to more than one line if they cross at that point. An established problem is to find upper bounds for the number of incidences between finite sets of points and lines of given cardinality. 

Specifically, fix a field $F$ and an integer $n$, and let $P$ and $L$ be finite sets of points and lines respectively in the plane $F \times F$ with $\left|P\right|$=$\left|L\right|=n$. Define $$I(P,L)=\left|\left\{(p,l)\in P \times L:p \in l\right\}\right|$$ to be the cardinality of the set of incidences between $P$ and $L$. The problem is to establish upper bounds on $I(P,L)$. A straightforward exercise in combinatorics \cite{TV} shows that one always has $I(P,L)\ll n^{\frac{3}{2}}$. So non-trivial incidence bounds are those of the form $I(P,L)\ll n^{\frac{3}{2}-\epsilon}$ for positive $\epsilon$. 

\subsection{Known bounds}

Different bounds are known for different choices of the field $F$. Things are largely settled in the settings $F=\mathbb{R}$ and $F=\mathbb{C}$. The result $\epsilon = 1/6$ was obtained in these settings, by Szem\'eredi and Trotter \cite{ST} and T\'oth \cite{toth} respectively. In both cases, the bound holds unconditionally and is sharp up to multiplicative constants.

Much less is known in the finite field setting $F=\mathbb{F}_q$. It is certainly not possible to have a non-trivial bound that holds in all cases, as the trivial bound $I(P,L)\approx n^{\frac{3}{2}}$ is achieved when $P=F \times F$ and $L$ is the set of lines determined by pairs of points in $P$. So one must impose some extra condition on $P$. 

When $F=\mathbb{F}_p$ is a finite field of prime order this can be simply a cardinality condition. The best-known result in this setting, due to Helfgott and Rudnev \cite{HR}, requires simply that $n$ is strictly less than $p$, and guarantees that $\epsilon\geq 1/10678$ when this condition is satisfied. This result is unlikely to be best-possible, and followed work of Bourgain, Katz and Tao \cite{BKT} which established the existence of a non-trivial $\epsilon>0$ so long as $n<p^{2-\delta(\epsilon)}$, but did not quantify it.

\subsection{Bounds over general finite fields}

The Helfgott-Rudnev bound is known only in $\mathbb{F}_p$, and so one would like to extend it to general (i.e. not necessarily prime) finite fields $\mathbb{F}_q$. In particular, it would be good to extend to $\mathbb{F}_{p^2}$, as this is the finite analogue of $\mathbb{C}$. However, general finite fields can have subfields, and so stronger conditions than just cardinality are required on $P$. This is because, as with the example above, if $K$ is a subfield of $F$ then the trivial bound $I(P,L)\approx n^{\frac{3}{2}}$ can be achieved when $P$ is the subplane $K \times K$.

It is therefore an interesting problem to find conditions on $P \subseteq \mathbb{F}_q \times \mathbb{F}_q$ for which an explicit Helfgott-Rudnev-type bound holds for any $L$ with $|L|=|P|$. Progress on this problem sheds light on the relationship between the algebraic structure of fields and the geometric structure of incidences. Ultimately one would like to find an algebraic condition for $P$ that is both necessary and sufficient for an explicit incidence bound.

The natural condition to try imposing on $P$ would be to insist that it is `not too close' to being a copy of a subplane, for example by ensuring that its projection onto one of either the $x$- or $y$-axis is `not too close' to a copy of a subfield. However, the currently-known approaches for proving Helfgott-Rudnev-type bounds rely on first applying a projective transformation to $P$, which could disrupt such a condition. So any condition must, additionally, be preserved by projective transformation. 

\subsection{Results}

We present an incidence result in $\mathbb{F}_q$, which holds so long as $P$ satisfies certain conditions. Informally, these are that the projection $A(P)$ of $P$ onto some co-ordinate axis has no more than `half-dimensional interaction' with `large' subfields $G$ of $\mathbb{F}_q$, where `large' will be defined relative to the cardinality $n=|P|$. 

By no more than `half dimensional interaction', we mean that $A(P)$ does not intersect an affine copy of $G$ in more than $|G|^{1/2}$ places, and intersects no more than $|G|^{1/2}$ distinct translates of $G$. Since the motivation is that such sets are a long way from being fields, we shall call them `antifields' and `strong antifields'.

\begin{definition}[Antifields]
Let $F$ be a field and $\lambda>0$.
\begin{enumerate}
\item 
Let $A \subseteq F$. Then
\begin{enumerate}
\item
$A$ is a $\mathbf{(1,\lambda)}$-\textbf{antifield} if $\left|A \cap (aG+b)\right|\leq \max\left\{\lambda,|G|^{\frac{1}{2}}\right\}$ for all subfields $G$ of $F$ and all $a,b \in F$.
\item
$A$ is a $\mathbf{(1,\lambda)}$-\textbf{strong-antifield} if it is a $(1,\lambda)$-antifield and, for every subfield $G$ with $|G|\geq \lambda$, it intersects strictly fewer than $\max\left\{\lambda,|G|^{\frac{1}{2}}\right\}/2$ distinct translates $G+b$ of $G$.
\end{enumerate}
\item Let $P \subset F \times F$. Then 
\begin{enumerate}
\item $P$ is a $\mathbf{(2,\lambda)}$\textbf{-antifield} if the set $\left\{x:(x,y) \in P\right\}$ is a $(1,\lambda)$-antifield
\item $P$ is a $\mathbf{(2,\lambda)}$\textbf{-strong-antifield} if the set $\left\{x:(x,y) \in P\right\}$ is a $(1,\lambda)$-strong-antifield
\end{enumerate}
\end{enumerate}
\end{definition}

Note that since one can always apply a change of basis, the projection can in fact be onto any vector multiple of $\mathbb{F}_q$. 

Parts $1.(a)$ and $2.(a)$ of the definition are motivated by work of Katz and Shen \cite{KS} generalising sum-product bounds in $\mathbb{F}_p$ to $\mathbb{F}_q$. Parts $1.(b)$ and $2.(b)$ are motivated by the need to avoid disruption by projective transformations. A key idea, which shall be seen later, is that certain projective images of a strong antifield will always be antifields.

We are now able to state the result:

\begin{theorem}\label{theorem:result}
There is an absolute constant $\gamma$ such that if $F$ is a finite field, $P$ and $L$ are sets of points and lines respectively in $F \times F$ with $|P|=|L|=n$, and $P$ is additionally a $\left(2,\gamma n^{\frac{2560}{6419}}\right)$-strong-antifield, then $I(P,L) \ll n^{\frac{3}{2}-\frac{1}{12838}}$.   
\end{theorem}

The majority of this paper is concerned with the proof of Theorem \ref{theorem:result}. But since it is not necessarily obvious that many point sets should satisfy the conditions of the theorem, we shall first show that it is easy to construct examples in the important cases $q=p^2$ and $q=p^4$. This is demonstrated by the following two corollaries; the first corollary demonstrates the requirement for limited interaction with subfields, and the second corollary demonstrates how one can ignore `small' subfields.

\begin{corollary}[Construction when $q=p^2$]\label{theorem:p^2} Let $P \subseteq \mathbb{F}_{p^2} \times \mathbb{F}_{p^2}$ with $|P|=n$, and define $A=A(P)=\left\{x:(x,y) \in P\right\}$. Let $t$ be a defining element of $\mathbb{F}_{p^2}$ over $\mathbb{F}_p$, so that $\mathbb{F}_{p^2}=\mathbb{F}_p + t \mathbb{F}_p$. Suppose that $|A|\ll p$ and that $A=\bigcup_{j \in J}A_j$
where $J \subseteq \mathbb{F}_p$ with $|J|\ll \max\left\{p^{\frac{1}{2}},n^{\frac{2560}{6419}}\right\}$, and $A_j \subseteq \mathbb{F}_p+jt$ with $|A_j|\ll \max\left\{p^{\frac{1}{2}},n^{\frac{2560}{6419}}\right\}$ for each $j \in J$. Then we have 
$I(P,L)\ll n^{\frac{3}{2}-\frac{1}{12838}} $ for all sets of lines $L$ in $\mathbb{F}_{p^2}\times \mathbb{F}_{p^2}$ with $|L|=n$.
\end{corollary}

\begin{proof}
We need to show that the hypotheses imply that $P$ is a $\left(2,\gamma n^{\frac{2560}{6419}}\right)$-strong-antifield. To do this, we first need to show that $P$ is simply a $\left(2,\gamma n^{\frac{2560}{6419}}\right)$-antifield. Note that the only sets of the form $a\mathbb{F}_p+b$ with $a, b \in \mathbb{F}_{p^2}$ are given by $\mathbb{F}_p +jt$ and $t\mathbb{F}_p +k$, where $j,k$ range over $\mathbb{F}_p$. Note further that $\left(\mathbb{F}_p +jt \right)\cap\left( t\mathbb{F}_p+k\right)=\left\{jt+k\right\}$. We know by assumption that $$|A \cap \left(\mathbb{F}_p+jt\right)|\ll \max\left\{p^{\frac{1}{2}},n^{\frac{2560}{6419}}\right\}$$ for each $j \in \mathbb{F}_p$. Observe that 
\begin{align*}
|A \cap \left(t\mathbb{F}_p+k\right)|= \sum_{j \in \mathbb{F}_p}\left|A \cap \left(t \mathbb{F}_p +k \right)\cap\left(\mathbb{F}_p+jt\right) \right|= \# \left\{j \in \mathbb{F}_p:\left|A \cap \left(\mathbb{F}_p+jt\right)\right|\right\}\leq |J| \ll \max\left\{p^{\frac{1}{2}},n^{\frac{2560}{6419}}\right\}.
\end{align*}
So we conclude that $P$ is a $\left(2,\gamma n^{\frac{2560}{6419}}\right)$-antifield. Since $|J| \ll \max\left\{p^{\frac{1}{2}},n^{\frac{2560}{6419}}\right\}$ it is also a $\left(2,\gamma n^{\frac{2560}{6419}}\right)$-strong-antifield, as required.
\end{proof}

\begin{corollary}[Construction when $q=p^4$]\label{theorem:p^4} Let $P \subseteq \mathbb{F}_{p^4} \times \mathbb{F}_{p^4}$ with $|P|=n\gg p^{\frac{6419}{2560}}$, and define $A=A(P)=\left\{x:(x,y) \in P\right\}$. Let $t$ be a defining element of $\mathbb{F}_{p^4}$ over $\mathbb{F}_{p^2}$, so that $\mathbb{F}_{p^4}=\mathbb{F}_{p^2} + t \mathbb{F}_{p^2}$. Suppose that $|A|\ll p^2$ and that $A=\bigcup_{j \in J}A_j$
where $J \subseteq \mathbb{F}_{p^2}$ with $|J|\ll \max\left\{p,n^{\frac{2560}{6419}}\right\}$, and $A_j \subseteq \mathbb{F}_p+jt$ with $|A_j|\ll \max\left\{p,n^{\frac{2560}{6419}}\right\}$ for each $j \in J$. Then we have 
$I(P,L)\ll n^{\frac{3}{2}-\frac{1}{12838}} $ for all sets of lines $L$ in $\mathbb{F}_{p^4}\times \mathbb{F}_{p^4}$ with $|L|=n$.
\end{corollary}

\begin{proof}
We need to show that the hypotheses imply that $P$ is a $\left(2,\gamma n^{\frac{2560}{6419}}\right)$-strong-antifield. Note that since $n\gg p^{\frac{6419}{2560}}$, we can ignore the subfield $\mathbb{F}_p$ and need check this only with respect to the subfields $\mathbb{F}_{p^2}$ and  $\mathbb{F}_{p^4}$. This checking follows Corollary \ref{theorem:p^2}.
\end{proof}

\section{Structure for proving Theorem \ref{theorem:result}}
The rest of the paper is concered with proving Theorem \ref{theorem:result}. This section outlines the structure of the proof. It states results, which will be proved later, and shows how they fit together to give the overall proof. There are two components to this. The first component is a key lemma that relates the algebraic and geometric structure of antifields. The second component uses this key lemma, and a method of Katz and Shen \cite{KS}, as part of an otherwise technical generalisation of the Helfgott-Rudnev proof.  

\subsection{The first component: Relating the algebraic and geometric stucture of antifields}
Recall that we defined both \textbf{antifields} and \textbf{strong-antifields}, that both are defined algebraically, and that Theorem \ref{theorem:result} is a statement about strong-antifields. The first component of the proof of Theorem \ref{theorem:result} is to relate the algebraic and geometric structure of these objects by showing that under certain projective transformations the image of a strong-antifield is an antifield.

The formal statement is expressed in terms of \textbf{cross ratios}. These are projective invariants, which means that they are preserved by projective transformations of a line and so are important in projective geometry.

\begin{definition}
Let $F$ be a field and let $a,b,c,d \in F$ with $a \neq d$ and $b \neq c$. Then define the \textbf{cross ratio} $X(a,b,c,d)$ by
$$X(a,b,c,d)=\frac{(a-b)(c-d)}{(a-d)(c-b)}$$
\end{definition}

We can now state the key lemma:

\begin{lemma}\label{theorem:key}
Let $A \subseteq F$ be a $(1,\lambda)$-strong-antifield and let $B \subseteq F$. Suppose there is a cross-ratio-preserving injection $\tau:B \to A$ (i.e. an injection $\tau$ for which $X(\tau(b_1),\tau(b_2),\tau(b_3),\tau(b_4))=X(b_1,b_2,b_3,b_4)$ whenever $b_1,b_2,b_3,b_4 \in B$). Then $B$ is a $(1,\lambda)$-antifield. 
\end{lemma}

\subsection{The second component: Applying the first component in a technical modification of the Helfgott-Rudnev proof}

The structure of the second component broadly follows \cite{HR}. It begins by applying Lemma \ref{theorem:key} in an adaptation of an argument of Bourgain, Katz and Tao \cite{BKT} to replace $L$ and $P$ with a construction of lines and points of a certain form, at the expense of some incidences and of passing from a strong-antifield to an antifield.

\begin{proposition}\label{theorem:propbelow}
Let $F$ be a field, and let $P$ and $L$ be a set of lines and points respectively in $F \times F$ with $|P|=|L|=n$ such that $I(P,L)=n^{\frac{3}{2}-\epsilon}$ for some $\epsilon>0$. Let $\lambda \geq 0$. Then, if $P$ is a $(2,\lambda)$-strong-antifield there exist:

\begin{enumerate}
\item Sets $A,B \subseteq F$ with $|A|,|B| \ll n^{\frac{1}{2}+\epsilon}$ and $0 \notin B$
\item A set $L_A$ of lines through the origin with gradients in $A$.
\item A set $L_B$ of horizontal (i.e. gradient $0$) lines with $y$-intercepts in $B$
\item A $(2,\lambda)$-antifield $P^*$ with $|P^*|\leq n$, the points of which each lie on the intersection of a line in $L_A$ with a line in $L_B$.
\end{enumerate}
such that $I\left(P^*,L(P^*)\right)\gg n^{\frac{3}{2}-5\epsilon}$
where $L(P^*)$ is the set of lines determined by pairs of points in $P^*$.
\end{proposition}

Following \cite{HR} we then generalise the definition of incidences to colinear $k$-tuples for any integer $k$:

\begin{definition}[Colinear $k$-tuples]\label{theorem:colinear}
Let $F$ be a field. Let $P$ be a finite set of points in $F \times F$ and let $L$ be a finite set of lines in $F \times F$. We define the number of \textbf{colinear $k$-tuples} between $P$ and $L$, denoted $I_k(P,L)$ by
$$I_k(P,L)=\left|\left\{\left(p_1,\ldots,p_k,l\right)\in P^k \times L:p_1,\ldots,p_k \in l\right\}\right|$$
\end{definition}

This generalises the definition of incidences because $I(P,L)=I_1(P,L)$. Moreover, the following lemma shows that H\"older's inequality relates incidences to colinear $k$-tuples:

\begin{lemma}\label{theorem:relation}
Let $F$ be a field and $k \in \mathbb{N}$. Let $P,L$ be sets of points and lines in $F \times F$. Then we have $I_k\left(P,L\right) \geq \frac{I\left(P,L\right)^k}{|L|^{k-1}}$.
\end{lemma}

\begin{proof}
Define $f:L \to \mathbb{N}$ by $f(l)=\sum_{p \in P}\delta_{lp}$ where $\delta_{lp}=1$ if $p \in L$ and $0$ otherwise, i.e. $f(l)$ is the number of points in $P$ that are incident to $l$. Note that $\left\|f \right\|_k=I_k(P,L)^{\frac{1}{k}}$. H\"older's inequality implies that $\left\|f \right\|_1 \leq \left\|f \right\|_k \left\|1 \right\|_{\frac{k}{k-1}}$, which is the same as $I(P,L) \leq I_k(P,L)^{\frac{1}{k}} |L|^{\frac{k-1}{k}} $.
\end{proof}

Applying Lemma \ref{theorem:relation} with $k=3$ reinterprets Proposition \ref{theorem:propbelow} as a lower bound on colinear triples:

\begin{corollary}\label{theorem:below}
With the notation in Proposition \ref{theorem:propbelow} and Definition \ref{theorem:colinear}, we also have 
$I_3\left(P^*,L(P^*)\right) \gg n^{\frac{5}{2}-15\epsilon}$
\end{corollary}

So we have a lower bound on colinear triples in $P^*$. Separately, the next proposition gives an upper bound on this quantity, which is obtained by combinatorial methods. Its proof uses the method in \cite{KS} to adapt the approach in \cite{HR}.

\begin{proposition}\label{theorem:above}
There is an absolute constant $\gamma_1$ such that if: 
\begin{itemize}
\item $F$ is a field and $A$, $B$ are finite subsets of $F$ with $0 \notin B$.
\item $L_A$ is the set of lines through the origin with gradients lying in $A$.
\item $L_B$ is the set of horizontal lines crossing the $y$-axis at some $b \in B$.
\item $P$ is a set of points, each lying on the intersection of some line in $L_A$ with some line in $L_B$.
\item $T:=I_3\left(P,L(P)\right)$.
\item $P$ is, additionally, a $\left(2,\frac{\gamma_1 T^{65}}{|A|^{130}|B|^{194}}\right)$-antifield.
\end{itemize}

Then:

\begin{equation*}
T\ll \max{\left\{\left|A\right|^{\frac{643}{321}}\left|B\right|^{\frac{961}{321}},\left|A\right|^{\frac{535}{267}}\left|B\right|^{\frac{799}{267}},\left|A\right|^{\frac{499}{249}}\left|B\right|^{\frac{743}{249}} \right\}} 
\end{equation*}
\end{proposition}

The results collected above then allow us to prove Theorem \ref{theorem:result}:

\paragraph{Proving Theorem \ref{theorem:result} from the propositions}
Let $|P|=|L|=n$ with $I(P,L)=n^{\frac{3}{2}-\epsilon}$. If $\epsilon > 1/12838$ then we are already done, so assume that $\epsilon \leq 1/12838$. We shall find a constant $\gamma$ such that $\epsilon \geq 1/12838$ so long as $P$ is a $\left(2,\gamma n^{\frac{1}{2}-\frac{1299}{12838}}\right)$-strong-antifield.

So let us suppose that $P$ is a $\left(2,\gamma n^{\frac{1}{2}-\frac{1299}{12838}}\right)$-strong-antifield, where $\gamma$ is a constant to be specified. Apply Proposition \ref{theorem:propbelow} and Corollary \ref{theorem:below} to obtain a particular $\left(2,\gamma n^{\frac{1}{2}-\frac{1299}{12838}}\right)$-antifield $P^*$ for which 
\begin{equation}\label{eq:below}
T:=I_3\left(P^*,L(P^*)\right)\gg n^{\frac{5}{2}-15 \epsilon}
\end{equation}

and for which Proposition \ref{theorem:above} is applicable so long as 
\begin{equation}\label{eq:req}
\gamma n^{\frac{1}{2}-\frac{1299}{12838}} \leq \frac{\gamma_1T^{65}}{|A|^{130}|B|^{194}}
\end{equation}
where $\gamma_1$ is an absolute constant. Note also that
\begin{equation}\label{eq:abound}
|A|,|B|\ll n^{\frac{1}{2}+\epsilon}
\end{equation}
Now, since $\epsilon \leq 1/12838$ and combining \eqref{eq:below} and \eqref{eq:abound}, we see that there is an absolute constant $\gamma_2$ such that $$n^{\frac{1}{2}-\frac{1299}{12838}}\leq n^{\frac{1}{2}-1299\epsilon} \leq \gamma_2 \frac{T^{65}}{|A|^{130}|B|^{194}}$$ So we can ensure that \eqref{eq:req} holds by taking $\gamma=\frac{\gamma_1}{\gamma_2}$. We therefore have by Proposition \ref{theorem:above} that
\begin{equation}\label{eq:above}
T\ll \max{\left\{\left|A\right|^{\frac{643}{321}}\left|B\right|^{\frac{961}{321}},\left|A\right|^{\frac{535}{267}}\left|B\right|^{\frac{799}{267}},\left|A\right|^{\frac{499}{249}}\left|B\right|^{\frac{743}{249}} \right\}} 
\end{equation}
Comparing \eqref{eq:below} and \eqref{eq:above}, plugging in \eqref{eq:abound}, and taking logs then yields $\epsilon\geq 1/12838$ as required. 

\subsection{The rest of this paper}

The proof of Theorem \ref{theorem:result} will be complete once Propositions \ref{theorem:propbelow} and \ref{theorem:above} have been established. Lemma Lemma \ref{theorem:key} is used for proving Propositions \ref{theorem:propbelow}. The proofs of these three results are the subject of the rest of the paper:

\begin{itemize}
\item Section \ref{section:aflemma} presents the proof of Lemma \ref{theorem:key}
\item Section \ref {section:BKTproof} presents the proof of Proposition \ref{theorem:propbelow}.
\item Section \ref{section:lemmata} collects some technical lemmata that will be useful when proving Proposition \ref{theorem:above}, some with proof and some without.
\item Finally, Section \ref{section:above} presents the proof of Proposition \ref{theorem:above}.
\end{itemize}

\section{Proving Lemma \ref{theorem:key}}\label{section:aflemma}
This section is concerned the proof of Lemma \ref{theorem:key}. Recall the statement of the lemma:

\begin{tabular}{|p{15cm}|}
\hline
\paragraph{Lemma \ref{theorem:key}}
Let $A \subseteq F$ be a $(1,\lambda)$-strong-antifield and let $B \subseteq F$. Suppose there is a cross-ratio-preserving injection $\tau:B \to A$ (i.e. an injection $\tau$ for which $X(\tau(b_1),\tau(b_2),\tau(b_3),\tau(b_4))=X(b_1,b_2,b_3,b_4)$ whenever $b_1,b_2,b_3,b_4 \in B$). Then $B$ is a $(1,\lambda)$-antifield. 
\\
\space
\\
\hline
\end{tabular}

For a set $A$, define $X(A)=\left\{X(a,b,c,d):a,b,c,d \in A, a \neq d, b \neq c\right\}$. To prove Lemma \ref{theorem:key} we will need the following intermediate result:

\begin{lemma}\label{theorem:cr}
Let $F$ be a field. Suppose $A \subseteq F$ and there is a subfield $G$ of $F$ for which $X(A)\subseteq G $. Then either $|A \cap (xG+y)| \leq 2$ for all $x,y \in F$, or there exist $x,y \in F$ such that $A \subseteq xG+y$.
\end{lemma}

\begin{proof}
We show that if $|A \cap (xG+y)| \geq 3$ then $A \subseteq xG+y$. Let $a,b,c$ be three distinct elements of  $A \cap \left(xG+y\right)$ and suppose for a contradiction that $A \nsubseteq xG+y$. Then we can find $d \in A$ with $d \notin xG+y$. So we have 
\begin{align*}
a=g_1x+y\\
b=g_2x+y\\
c=g_3x+y\\
d=g_4x+z
\end{align*}
where $g_1,g_2,g_3,g_4 \in G$ and $\frac{z-y}{x} \notin G$. Moreover, since $a,b,c$ are distinct, we know that $g_1,g_2,g_3$ are distinct. Finally, we know that $a,b,c \neq d$. We then know by assumption that 

$$\frac{(a-b)(c-d)}{(a-d)(c-b)}\in G$$

But we also have

\begin{align*}
\frac{(a-b)(c-d)}{(a-d)(c-b)}&=\frac{x(g_1-g_2)(x(g_3-g_4)+(y-z))}{(x(g_1-g_4)+(y-z))(x(g_3-g_2))}= \left(\frac{g_1-g_2}{g_3-g_2}\right) \frac{g_3-g_4 + \frac{y-z}{x}}{g_1-g_4 + \frac{y-z}{x}}
\end{align*}

Since $g_1,g_2$ and $g_3$ are distinct, this means that $$\frac{g_3-g_4 + \frac{y-z}{x}}{g_1-g_4 + \frac{y-z}{x}} \in G$$ and so there exists $g_5 \in G$ with $$\frac{g_3-g_4 + \frac{y-z}{x}}{g_1-g_4 + \frac{y-z}{x}} = g_5$$

We now split into two cases, according to whether or not $g_5=1$. If $g_5=1$ then we obtain $g_3=g_1$, which contradicts the fact that these two elements are distinct. If $g_5 \neq 1$ then we obtain  $$\frac{y-z}{x}=\frac{g_5(g_1-g_4)-g_3+g_4}{1-g_5} \in G $$ which contradicts the fact that $\frac{y-z}{x}\notin G$. Either way, we are done.
\end{proof}

\begin{corollary}\label{theorem:coroll}
Let $F$ be a field, $G$ be a subfield of $F$, $A \subseteq F$ be a $(1,\lambda)$-strong-antifield, and $A' \subseteq A$ be such that $|A'|\geq \max\left\{\lambda,|G|^{\frac{1}{2}}\right\}$. Then $X(A')\nsubseteq G$. 
\end{corollary}

\begin{proof}
Suppose that there exists $A' \subseteq A$ with $|A'| \geq \max \left\{\lambda,|G|^{\frac{1}{2}}\right\}$ and $X(A') \subseteq G$. Then by Lemma \ref{theorem:cr}, either $A' \subseteq aG+b$ for some $a,b \in F$, or $\left|A' \cap (aG+b)\right|\leq 2$ for all $a,b \in F$. 

In the former case, we have $A' \subseteq A \cap \left(aG+b\right)$ and so $\left|A \cap \left(aG+b\right)\right|\geq \max \left\{\lambda,|G|^{\frac{1}{2}}\right\}$. In the latter case we have $|A' \cap (G+b)| \leq 2$ for all distinct translates $G+b$ of $G$, which means that $A'$ and therefore $A$ intersects at least $\max\left\{\lambda,|G|^{\frac{1}{2}}\right\}/2$ such translates. 

Either way, we contradict the fact that $A$ is a $(1,\lambda)$-strong-antifield and are therefore done.
\end{proof}

We are now in a position to prove Lemma \ref{theorem:key}.

\paragraph{Proof of Lemma \ref{theorem:key}}
Suppose for a contradiction that there is a subfield $G$ of $F$ and elements $a,b \in F$ such that $$\left|B \cap (aG+b)\right|\geq \max\left\{\lambda,|G|^{\frac{1}{2}}\right\}$$
Let $B'=B \cap (aG+b)$. Then we have $\tau(B')\subseteq A$ and $|\tau(B')|=|B'|\geq \max\left\{\lambda,|G|^{\frac{1}{2}}\right\}$, but also $X(\tau(B'))=X(B')\subseteq G$. This contradicts Corollary \ref{theorem:coroll} and so we are done. This completes the proof of Lemma \ref{theorem:key}.

\section{Proof of Proposition \ref{theorem:propbelow}}\label{section:BKTproof}
We will now use Lemma \ref{theorem:key} to prove Proposition \ref{theorem:propbelow}.
Recall the statement of Proposition \ref{theorem:propbelow}:

\begin{tabular}{|p{15cm}|}
\hline
\paragraph{Proposition \ref{theorem:propbelow}}
Let $F$ be a field, and let $P$ and $L$ be a set of lines and points respectively in $F \times F$ with $|P|=|L|=n$ such that $I(P,L)=n^{\frac{3}{2}-\epsilon}$ for some $\epsilon>0$. Let $\lambda \geq 0$. Then, if $P$ is a $(2,\lambda)$-strong-antifield there exist:

\begin{enumerate}
\item Sets $A,B \subseteq F$ with $|A|,|B| \ll n^{\frac{1}{2}+\epsilon}$ and $0 \notin B$
\item A set $L_A$ of lines through the origin with gradients in $A$.
\item A set $L_B$ of horizontal (i.e. gradient $0$) lines with $y$-intercepts in $B$
\item A $(2,\lambda)$-antifield $P^*$ with $|P^*|\leq n$, the points of which each lie on the intersection of a line in $L_A$ with a line in $L_B$.
\end{enumerate}
such that $$I\left(P^*,L(P^*)\right)\gg n^{\frac{3}{2}-5\epsilon}$$
where $L(P^*)$ is the set of lines determined by pairs of points in $P^*$.
\\
\space
\\
\hline
\end{tabular}

Recall that for a point $p$ and a line $l$ we define $\delta_{pl}$ to be $1$ if $p \in l$ and $0$ otherwise. We initially follow \cite{BKT} and \cite{HR}.

The first step is to show that we may assume every point in $P$ is incident to $\gg n^{\frac{1}{2}-\epsilon}$ and $\ll n^{\frac{1}{2}+\epsilon}$ lines in $L$. Indeed, let $P_+=\left\{p \in P: p \text{ is incident to} \geq 4n^{\frac{1}{2}+\epsilon} \text{ lines } l \in L\right\}$. Then: 
\begin{align*}
I\left(P_+,L\right)&=\sum_{p \in P_+}\sum_{l \in L}\delta_{pl}
\leq \frac{1}{4 n^{\frac{1}{2}+\epsilon}}\sum_{p \in P_+} \left(\sum_{l \in L}\delta_{pl}\right)^2
=\frac{1}{4 n^{\frac{1}{2}+\epsilon}}\sum_{l,l' \in L}\sum_{p \in P_+}\delta_{pl}\delta_{pl'}
\leq\frac{n^{\frac{3}{2}-\epsilon}}{2}
\end{align*}

Similarly, let $P_-=\left\{p \in P: p \text{ is incident to} \leq \frac{n^{\frac{1}{2}-\epsilon}}{3} \text{ lines } l \in L \right\}$. Then: 
\begin{align*}
I\left(P_-,L\right)=\sum_{p \in P_-}\sum_{l \in L}\delta_{pl}
\leq \sum_{p \in P_-}\frac{n^{\frac{1}{2}-\epsilon}}{3}
\leq \frac{n^{\frac{3}{2}-\epsilon}}{3}
\end{align*}

So between them $P_+$ and $P_-$ contribute only five sixths of the $n^{\frac{3}{2}-\epsilon}$ incidences. Without loss of generality we shall discard them and assume from now on that $|P| \leq n$, and that every point $p \in P$ is incident to $\gg n^{\frac{1}{2}-\epsilon}$ and $\ll n^{\frac{1}{2}+\epsilon}$ lines in $L$.

Let $L_1$ be the set of ``rich'' lines in $L$ defined by 
$$L_1=\left\{l \in L:l \text{ is incident to} \geq \frac{n^{\frac{1}{2}-\epsilon}}{20}\text{ points } p \in P\right\}$$

Let $P_1$ be the set of points in $P$ that are ``bushy'' relative to $L_1$, defined by 
$$P_1=\left\{p \in P: p \text{ is incident to } \geq \frac{n^{\frac{1}{2}-\epsilon}}{20}\text{ lines in } L_1\right\}$$ 

We need to check that $P_1$ is non-empty. Note firstly that 

$$I(P,L\backslash L_1)=\sum_{p \in P}\sum_{l \in L \backslash L_1}\delta_{pl}\leq \sum_{l \in L \backslash L_1}\frac{n^{\frac{1}{2}-\epsilon}}{20} \leq \frac{n^{\frac{3}{2}-\epsilon}}{20} $$

and therefore $I(P,L_1)\gg I(P,L)$. Now note that

$$
I(P \backslash P_1,L_1)= \sum_{p \in P \backslash P_1} \sum_{l \in L_1}\delta_{pl}
< \sum_{p \in P \backslash P_1}\frac{n^{\frac{1}{2}-\epsilon}}{20}
\leq \frac{n^{\frac{3}{2}-\epsilon}}{20}
$$

This means that $I(P_1,L_1)\gg I(P,L_1)\gg I(P,L)$ and so $P_1$ is certainly non-empty. Now for each $p \in P_1$ let $P_p$ be the set of points in $P$ that are joined to $p$ by a line in $L_1$. We have:
\begin{align*}
\left|P_p \right|&=\sum_{q \in P}\sum_{l \in L_1}\delta_{pl}\delta_{ql}
=\sum_{l \in L_1}\delta_{pl}\sum_{q \in P}\delta_{ql}
\gg n^{\frac{1}{2}-\epsilon}\sum_{l \in L_1}\delta_{pl}
\gg n^{1-2\epsilon}
\end{align*}
This means that:
\begin{align*}
\left|P_1\right|n^{1-2\epsilon}\ll \sum_{p \in P_1}\left|P_p\right|
\leq \sqrt{\left|P_1\right|}\sqrt{\sum_{p,q \in P_1}\left|P_p \cap P_q\right|}
\end{align*}
where the second inequality follows by Cauchy-Schwartz. So we have:

\begin{equation}\label{equat:cspigeon}
\left|P_1\right|n^{2-4\epsilon}\ll \sum_{p,q \in P_1}\left|P_p \cap P_q\right|
\end{equation}

For each $p \in P$ define $x_p$ to be the $x$-co-ordinate of $p$. And for each $x \in F$ define $P^x=\left\{p \in P:x_p=x\right\}$. It is easy to see that $|P^x|n^{\frac{1}{2}-\epsilon} \ll I(P^x,L) \leq 2n$ and so we deduce that $|P^x|\ll n^{\frac{1}{2}+\epsilon}$ for every $x \in F$. Plugging this into \eqref{equat:cspigeon} yields

$$|P_1|n^{2-4 \epsilon} \ll \sum_{p,q \in P_1:x_p \neq x_q}\left|P_p \cap P_y\right|+ \sum_{p \in P_1}\sum_{q \in P^{x_p}}\left|P_p \cap P_q\right| \ll \sum_{p,q \in P_1:x_p \neq x_q}\left|P_p \cap P_y\right|+ |P_1|n^{\frac{3}{2}+\epsilon}$$

We can therefore fix two distinct points $p,q \in P_1$ with $x_p\neq x_q$ such that 

\begin{align*}
\left|P_{p} \cap P_{q}\right|&\gg \frac{n^{2-4\epsilon}}{\left|P\right|}
\gg n^{1-4\epsilon}
\end{align*}

Now let $P'=P_{p} \cap P_{q}$ and note that

\begin{align*}
I(P',L)=\sum_{p \in P'}\sum_{l \in L}\delta_{pl} \geq \left|P'\right|n^{\frac{1}{2}-\epsilon}\gg n^{\frac{3}{2}-5 \epsilon}
\end{align*}

Since $I(P^{x_p},L)\leq n$ we can discard all points in $P^{x_p}$ other than $p$ , and thereby assume $P^{x_p}=\left\{p\right\}$. 

At this point we diverge from \cite{BKT} and \cite{HR}. All we shall carry forward are the facts that:

\begin{enumerate}
\item $I(P',L)\gg n^{\frac{3}{2}-5 \epsilon}$.
\item $P'$ is a $(2,\lambda)$-strong-antifield.
\item There are two points $p,q$, lying on distinct vertical lines, such that $P'=P_p \cap P_q$ where $P_p$ is a set of points lying on $O(n^{\frac{1}{2}+\epsilon})$ lines through $p$, and $P_q$ is a set of points lying on $O(n^{\frac{1}{2}+\epsilon})$ lines through $q$
\item No point in $P'$ lies on the vertical line through $p$.
\end{enumerate}

These facts are unaffected by translation of $P'$ and so without loss of generality we shall assume that $p$ is in fact the origin. 

Recall that the \textbf{projective plane} $\mathbb{P}^2(F)$ is defined to be $F^3 \backslash \left(0,0,0\right)$, modulo dilations. We embed $F \times F$ in $\mathbb{P}^2(F)$ by identifying $(x,y) \in F \times F$ with $(x,y,1) \in \mathbb{P}^2(F)$. This accounts for all elements of $\mathbb{P}^2(F)$ apart from those of the form $(x,y,0)$; these are said to lie on the \textbf{line at infinity}. For our purposes, the only such point we need consider is the point $(1,0,0)$. Every line incident to this point has gradient 0, and is therefore horizontal. A \textbf{projective transformation} is an invertible linear map from $\mathbb{P}^2(F)$ to itself, i.e. a $3 \times 3$ non-singular matrix, and has the important property that it maps points to points and lines to lines.

Returning to the proof, we apply the projective transformation $\tau$ given by
\begin{equation*}
\tau=\left(
\begin{array}{ccc}
0&0&1\\
0&1&0\\
1&0&0	
\end{array}\right)
\end{equation*}

Note that:
\begin{enumerate}

\item $I(\tau(P'), L(\tau(P')))\geq I(\tau(P'),\tau(L))=I(P',L)\gg n^{\frac{3}{2}-5 \epsilon}$

\item $\tau$ maps the $y$-axis to the line at infinity. In particular, it maps the origin (which we have assumed to be $p$) to the point at infinity with gradient $0$, and so the points in $\tau(P_p)$ lie on $O(n^{\frac{1}{2}+\epsilon})$ horizontal lines. 

\item Since $P'$ has no points on the $y$-axis, the image $\tau(P')$ is contained in $F \times F$.

\item Since $q$ does not lie on the $y$-axis, the point $\tau(q)$ lies in $F \times F$ and not the line at infinity. Every point in $\tau(P_q)$ lies on one of $O(n^{\frac{1}{2}+\epsilon})$ lines through $\tau(q)$.  

\item $\tau(x,y)=\left(\frac{1}{x},\frac{y}{x}\right)$ for each point $(x,y)$ with $x \neq 0$. So the map $x \mapsto x^{-1}$ is a cross-ratio-preserving injection from $\left\{x:(x,y) \in \tau(P') \right\}$ to $\left\{x:(x,y) \in P' \right\}$. Since $P'$ is a $(2,\lambda)$-strong-antifield, Lemma \ref{theorem:key} implies that $\tau(P)$ is a $(2,\lambda)$-antifield.
\end{enumerate}

From the above we see that we have a $(2,\lambda)$-antifield $P^*=\tau(P')$ such that:

\begin{enumerate}
\item $I(P^{*},L(P^*))\gg n^{\frac{3}{2}-5 \epsilon}$
\item Each point in $P^{*}$ lies on
\begin{enumerate}
\item one of $O(n^{\frac{1}{2}+\epsilon})$ lines that pass through a single point $s$ in $F \times F$.
\item one of $O(n^{\frac{1}{2}+\epsilon})$ horizontal lines.
\end{enumerate}
\end{enumerate}

The properties above are again invariant under translation and so without loss of generality we may assume that $s$ is the origin. And since each horizontal line in $P^*$ contributes at most $n$ incidences we can discard points to assume that $0 \notin B$. We then take $A$ to be the set of gradients of the $O(n^{\frac{1}{2}+\epsilon})$ lines through the origin, and $B$ to be the $y$-intercepts of the $O(n^{\frac{1}{2}+\epsilon})$ horizontal lines. This completes the proof of the proposition. 

\section{Lemmata for proving Proposition \ref{theorem:above}}\label{section:lemmata}
This section collects the technical lemmata that will be used to prove Proposition \ref{theorem:above}. 

\subsection{Pivoting results}

We will make use of some `pivoting' results. The first, Lemma \ref{theorem:pivot2}, was applied in the Helfgott-Rudnev proof \cite{HR}, and before that in e.g. \cite{GK}, \cite{garaev}, \cite{KSprime}, \cite{shen} and \cite{li}. It is stated here without proof.

\begin{lemma}[Pivoting lemma 1]\label{theorem:pivot2}
Let $F$ be a field, let $Z \subseteq F$ and let $R(Z)=\frac{Z-Z}{Z-Z}$. Let $a,b \in F$. Then if $\left|R(Z)\right|\geq \left|Z\right|^2$ there exist $z_1,z_2,z_3,z_4 \in aZ+b$ such that for all $Z' \subseteq Z$ with $\left|Z'\right|\gg \left|Z\right|$ we have $\left|Z\right|^2 \approx \left|\left(z_1-z_2\right)Z'+\left(z_3-z_4\right)Z'\right|$
\end{lemma}

The next lemma is a quick and well-known result that is a necessary tool for the lemma that follows it:

\begin{lemma}
Let $F$ be a field, let $Z \subseteq F$ and let $R(Z)=\frac{Z-Z}{Z-Z}$. If $x \notin R(Z)$ then $\left|Z+xZ\right|\approx |Z|^2$.
\end{lemma} 
\begin{proof}
Clearly $\left|Z+xZ\right|\ll |Z|^2$, so we seek $\left|Z+xZ\right|\gg |Z|^2$. If there exist $z_1,z_2,z_3,z_4 \in Z$ with $z_2 \neq z_4$ and $z_1+xz_2=z_3+xz_4$, then we can write $x=\frac{z_1-z_3}{z_2-z_4}$, which contradicts the fact that $x \notin R(Z)$. So there is only one way of writing each elemnent $v\in Z+ xZ$ in the form $v=z_1+xz_2$ with $z_1,z_2 \in Z$. We therefore have $|Z+xZ|=\frac{|Z|\left(|Z|-1\right)}{2}\gg|Z|^2$, as required.
\end{proof}

Lemma \ref{theorem:pivot1}, due to Katz and Shen \cite{KS}, generalises an approach that is traditionally used in conjunction with Lemma \ref{theorem:pivot2}. The generalistation means that the result allows for the possibility of nontrivial additive subgroups.

\begin{lemma}[Pivoting lemma 2]\label{theorem:pivot1}
Let $F$ be a field and let $Z \subseteq F$ be finite such that $R(Z)=\frac{Z-Z}{Z-Z}$ is not a subfield of $F$. Let $a,b \in F$. Then either

\begin{enumerate}
\item \textbf{$\mathbf{R(aZ+b)}$ is not closed under multiplication}, in which case there exist $x_1,x_2,z_1,z_2,z_3,z_4 \in Z$ such that $\left|Z'\right|^2\leq\left|x_1\left(z_1-z_2\right)Z'-x_2\left(z_1-z_2\right)Z'+x_1\left(z_3-z_4\right)Z'\right|$ for all $Z' \subseteq Z$.

\item \textbf{$\mathbf{R(aZ+b)}$ is closed under multiplication but is not closed under addition}, in which case there exist $y_1,y_2,y_3,y_4 \in Z$ such that  $\left|Z'\right|^2 \leq \left|\left(y_1-y_2\right)Z'+\left(y_3-y_4\right)Z'+\left(y_3-y_4\right)Z' \right|$ for all $Z' \subseteq Z$. 

\end{enumerate}
\end{lemma}

\begin{proof}
Note that $R(aZ+b)=R(Z)$ so without loss of generality we may assume $a=1$ and $b=0$. 

\paragraph{Case 1}
Since $R(Z) \cdot R(Z) \neq R(Z)$ there are $x_1,x_2,x_3,x_4,y_1,y_2,y_3,y_4 \in Z$ with
$$\frac{x_1-x_2}{x_3-x_4}\frac{y_1-y_2}{y_3-y_4} \notin R(Z) $$
This can be written as 
$$\frac{x_1-x_2}{x_1}\frac{x_1}{x_1-x_3}\frac{x_1-x_3}{x_4}\frac{x_4}{x_3-x_4}\frac{y_1-y_2}{y_3-y_4} \notin R(Z)$$
and so there are $a_1,a_2,b_1,b_2,b_3,b_4 \in Z$ with $\frac{a_1-a_2}{a_1}\frac{b_1-b_2}{b_3-b_4} \notin R(Z)$. We therefore have that for any $Z' \subseteq Z$

\begin{align*}
\left|Z'\right|^2&\approx \left|Z'+\frac{a_1-a_2}{a_1}\frac{b_1-b_2}{b_3-b_4}Z'\right|\leq \left|a_1(b_1-b_2)Z' -a_2(b_1-b_2)Z' +a_1(b_3-b_4)Z'\right|
\end{align*}
This completes the proof of Case 1.

\paragraph{Case 2}
We seek $z_1,z_2,z_3,z_4 \in Z$ such that $\frac{z_1-z_2}{z_3-z_4}+1 \notin R(Z)$. We will then be done, as for any $Z' \subseteq Z$ we will have
\begin{align*}
\left|Z'\right|^2&\approx\left|Z'+\left(\frac{x_1-x_2}{x_3-x_4}+1\right)Z'\right|
\leq\left|(x_1-x_2)Z' + (x_3-x_4)Z' + (x_3-x_4)Z' \right|
\end{align*}
Since $R(Z)+R(Z) \neq R(Z)$ there are $x_1,x_2,x_3,x_4,y_1,y_2,y_3,y_4 \in Z$ with 
$$\frac{x_1-x_2}{x_3-x_4}+\frac{y_1-y_2}{y_3-y_4} \notin R(Z) $$
On the other hand, since $R(Z) \cdot R(Z)=R(Z)$ there are $z_1,z_2,z_3,z_4 \in Z$ with 
$$\frac{x_1-x_2}{x_3-x_4}\frac{y_3-y_4}{y_1-y_2}=\frac{z_1-z_2}{z_3-z_4}$$
Combining these two facts gives:
\begin{align*}
\frac{z_1-z_2}{z_3-z_4}+1&=\frac{x_1-x_2}{x_3-x_4}\frac{y_3-y_4}{y_1-y_2}+1
=\frac{y_3-y_4}{y_1-y_2}\left(\frac{x_1-x_2}{x_3-x_4}+\frac{y_1-y_2}{y_3-y_4}\right)
\notin R(Z)
\end{align*}

This completes the proof of Case 2 and therefore of the lemma.
\end{proof}

We will also use the following lemma, due to Katz and Shen. A proof can be found in \cite{KS}. 

\begin{lemma}\label{theorem:ratiofield}
If $R(Z) \subseteq G$ for some subfield $G$ of $F$, then $Z \subseteq aG+b$ for some $a,b \in F$
\end{lemma}

\subsection{A lemma about sumsets}

The following lemma was used in the Helfgott-Rudnev paper \cite{HR}, and is originally due to Bourgain \cite{bourgain}:

\begin{lemma}\label{theorem:intersection}
Let $F$ be a field. Let $X$ and $Y$ be finite subsets of $F$ and let $K=\max_{y \in Y} \left|X+yX \right|$
Then there exist elements $x_1, x_2, x_3 \in X$ such that $\left|\left(X-x_1\right) \cap \left(x_2-x_3\right)Y \right|\gg\frac{\left|Y \right|\left|X \right|}{K}$.
\end{lemma}

\begin{proof}
Let $E$ be the number of solutions to the equation $x_1+y x_2 = x_3+yx_4 $ with $x_1,x_2,x_3,x_4 \in X$ and $y \in Y$. Then

\begin{align*}
E=\sum_{y \in Y}\sum_{k \in X+yX}\left|X \cap \left( \frac{X-k}{y}\right) \right|^2
\geq \sum_{y \in Y}\frac{\left(\sum_{k \in X+yX}\left|X \cap \left(\frac{X-k}{y}\right) \right|\right)^2}{\left|X+yX\right|}
\geq \frac{\left|X\right|^4\left|Y\right|}{K}
\end{align*}

So there exist $z_1, z_2 \in X$ such that the equation $x_1 +yz_1=z_2+yx_2$  has $\gg\frac{\left|X \right|^2 \left|Y\right|}{K}$ solutions  $\left(x_1,x_2,y\right) \in X \times X \times Y$. In other words, if $X_1=X-z_1$ and $X_2=X-z_2$ then there are $\gg\frac{\left|X \right|^2 \left|Y\right|}{K}$ solutions $(u,v,y)\in X_1 \times X_2 \times Y$ to the equation $v=yu$. By averaging, there is an element $u_* =x_*-z_1\in X_1$ with $x_* \in X$ such that $v=y u_*$ has $\gg \frac{\left|Y\right|\left|X\right|}{K}$ solutions. Thus:
$$\left|\left(X-z_2\right)\cap\left(x_*-z_1\right)Y\right|=\left|X_2 \cap u_* Y\right|\gg \frac{\left|Y\right|\left|X\right|}{K} $$

\end{proof}

\subsection{Standard results from additive combinatorics}
We record some standard results from additive combinatorics. The first, below, formalises a common technique.

\begin{lemma}[Popularity pigeonholing]\label{theorem:popularity}
Let $X$ be a finite set and let $f:X \to \left[1,N \right]$ be a function. Then there is a subset $Y \subseteq X$ with $\left|Y \right|\gg \frac{\sum_{x \in X} f(x)}{N} $ such that for any $y \in Y$ we have $f(y) \gg \frac{\sum_{x \in X}f(x)}{\left|X \right|}$
\end{lemma}

\begin{proof}
Let $Y=\left\{x \in X: f(x) \geq \alpha \right\}$ where $\alpha=\frac{\sum_{x \in X}f(x)}{2\left|X \right|}$. We seek to show that $\left|Y \right|\gg \frac{\sum_{x \in X}f(x)}{N}$. We see this as follows:
\begin{align*}
\sum_{x \in X}f(x)= \sum_{x:f(x) \geq \alpha}f(x) + \sum_{x: f(x) < \alpha}f(x) \leq N\left|Y \right|+ \alpha \left|X \right|
\end{align*}
So we have
\begin{align*}
\left|Y \right|&\geq \frac{\sum_{x \in X}f(x) - \alpha \left|X \right|}{N}=\frac{\sum_{x \in X}f(x)}{2N}\gg \frac{\sum_{x \in X}f(x)}{N}
\end{align*}

\end{proof}
We will use the following form of the Pl\"unnecke-Ruzsa inequality, due to Ruzsa \cite{ruzsa}:

\begin{lemma}[Pl\"unnecke-Ruzsa inequality]\label{theorem:plunnecke}
Let $X,B_1, \ldots, B_k \subseteq \mathbb{F}_p$. Then $\left|\sum_{j=1}^k B_j \right|\ll\frac{\prod_{j=1}^k \left|X + B_j\right|}{\left|X \right|^{k-1}}$
\end{lemma}

The following lemma is a version of the Balog-Szemer\'edi-Gowers theorem. A proof can be found in \cite{TV}, but this appears to have a typographical error which leads to an exponent of $-4$, rather than the correct exponent of -5 below. See \cite{FS} for a proof yielding the exponent of $-5$.

\begin{lemma}[Balog-Szemeredi-Gowers]\label{theorem:BSG}
Let $X,Y$ be additive sets with $\left|X\right|=\left|Y\right|=n$. Suppose that there is a subset $G \subseteq X \times Y$ such that $\left|X +^{G} Y\right|<n$ and that $\left|G\right|=\alpha n^2$ for some $\alpha \in (0,1)$. Then there exist subset $X' \subseteq X$ and $Y' \subseteq Y$ with $\left|X'\right|,\left|Y'\right|\gg \alpha n$ such that 
$\left|X'+Y'\right| \ll \alpha^{-5}n $
\end{lemma}

A proof of the following `covering' result can be found in \cite{shen}.

\begin{lemma}[Covering lemma]
\label{theorem:covering}
Let G be a group and $B,C \subseteq G$ be finite. Let $\epsilon \in (0,1)$. Then the number of translates of $C$ required to cover $(1- \epsilon)\left|B \right|$ elements of $B$ is $O_{\epsilon}\left(\frac{\left|B+C\right|}{\left|C\right|}\right)$.
\end{lemma}

\section{Proof of Proposition \ref{theorem:above}}
\label{section:above}

Recall the statement of Proposition \ref{theorem:above}:

\begin{tabular}{|p{15cm}|}
\hline
\paragraph{Proposition \ref{theorem:above}}

There is an absolute constant $\gamma_1$ such that if: 
\begin{itemize}
\item $F$ is a field and $A$, $B$ are finite subsets of $F$ with $0 \notin B$.
\item $L_A$ is the set of lines through the origin with gradients lying in $A$.
\item $L_B$ is the set of horizontal lines crossing the $y$-axis at some $b \in B$.
\item $P$ is a set of points, each lying on the intersection of some line in $L_A$ with some line in $L_B$.
\item $T:=I_3\left(P,L(P)\right)$.
\item $P$ is, additionally, a $\left(2,\frac{\gamma_1 T^{65}}{|A|^{130}|B|^{194}}\right)$-antifield.
\end{itemize}

Then:

\begin{equation*}
T\ll \max{\left\{\left|A\right|^{\frac{643}{321}}\left|B\right|^{\frac{961}{321}},\left|A\right|^{\frac{535}{267}}\left|B\right|^{\frac{799}{267}},\left|A\right|^{\frac{499}{249}}\left|B\right|^{\frac{743}{249}} \right\}} 
\end{equation*}
\\
\space
\\
\hline
\end{tabular}

This section uses the results of Section \ref{section:lemmata} to prove Proposition \ref{theorem:above}.

\subsection{Structure of the proof}

We shall assume that $P$ is a $(2,\lambda)$ antifield for some $\lambda$, and then show that the conclusion of the Proposition follows when $\lambda \approx \frac{T^{65}}{|A|^{130}|B|^{194}}$.

The proof of Proposition \ref{theorem:above} uses the following three claims, whose proofs are deferred. Instead, we shall first see how they are applied to prove the proposition. The proofs of the claims then follow.

\begin{claim}\label{theorem:bsgclaim}

There is a subset $C \subseteq \mathbb{F}_q$ with $|C| \gg \frac{T^5}{|A|^{10}|B|^{14}}$ such that for each $c \in C$ there is a pair of $(1,\lambda)$-antifields $A_c^1,A_c^2 \subseteq F$ with
\begin{equation}\label{size}
|A_c^1|,|A_c^2| \gg \frac{T}{|A||B|^3}
\end{equation}
\begin{equation}\label{sumset}
\left|A_c^1+cA_c^2\right|\ll \frac{\left|A\right|^{11}\left|B\right|^{15}}{T^5}
\end{equation}
Moreover, there exists a particular element $c_* \in C$ such that, writing $A_*=A_{c_*}$, we have
\begin{equation}\label{intersection}
\left|A_c^1 \cap A_*^1\right|, \left|A_c^2 \cap A_*^2\right|\gg \frac{T^4}{|A|^7|B|^{12}}
\end{equation}
for all $c \in C$.
\end{claim}

\begin{claim}\label{theorem:sumsetclaim}
The following bounds hold for each $c \in C$
\begin{equation}\label{claim1}
\left|A_c^1 + A_c^1\right|,\left|A_c^2 + A_c^2\right|  \ll \frac{\left|A\right|^{23}\left|B\right|^{33}}{T^{11}}
\end{equation}
\begin{equation}\label{claim2}
\left|c_*A_c^2 + cA_c^2\right| \ll \frac{\left|A\right|^{59}\left|B\right|^{87}}{T^{29}}
\end{equation}
\begin{equation}\label{claim3}
\left|c_*A_*^2 + cA_c^2\right| \ll \frac{\left|A\right|^{83}\left|B\right|^{132}}{T^{44}}
\end{equation}
\begin{equation}\label{claim4}
\left|c_*A_*^2 + cA_*^2\right| \ll \frac{\left|A\right|^{119}\left|B\right|^{177}}{T^{59}}
\end{equation}
\end{claim}

\begin{claim}\label{theorem:coveringclaim}
There exists an integer $\Gamma$ with
\begin{equation}\Gamma \ll \frac{\left|A\right|^{48}\left|B\right|^{72}}{T^{24}} \label{gammabound}\end{equation}
such that given any $c \in \pm C$, $x \in \mathbb{F}_q$, and $D \subseteq A_*^2$, a constant proportion of $cD+x$ can be covered with $\Gamma$ translates of $A_*^1$
\end{claim}

\subsection{Proof of Proposition \ref{theorem:above}, assuming claims}\label{section:mainproof}

Apply Lemma \ref{theorem:intersection} with $X=A_*^2$, $Y = \frac{1}{c_*}C$ and, by inequality \eqref{claim4}, $K\ll\frac{\left|A\right|^{119}\left|B\right|^{177}}{T^{59}}$. This provides $a_1,a_2,a_3 \in A_*^2$ such that 
\begin{align*}
\left|\left(A_*^2-a_1\right) \cap \left(\frac{a_2-a_3}{c_*}\right)C\right|\gg \frac{\left|A_*^2\right|\left|B_2'\right|}{K}
\gg\frac{T^{65}}{\left|A\right|^{130}\left|B\right|^{194}} 
\end{align*}
For convenience, define $Z=\left(A_*^2-a_1\right) \cap \left(\frac{a_2-a_3}{c_*}\right)C$, to give the lower bound
\begin{equation}\label{zbound}
\left|Z\right|\gg \frac{T^{65}}{\left|A\right|^{130}\left|B\right|^{194}} 
\end{equation}
We seek an upper bound for $\left|Z\right|$ with which to compare \eqref{zbound}. There are three possible cases: 
 
\begin{enumerate}

\item $\mathbf{R(Z)}$\textbf{ is not closed under multiplication}. By Lemma \ref{theorem:pivot1} there are then elements $c_1,c_2,d_1,d_2,d_3,d_4 \in C$ such that for every $Z' \subseteq Z$ with $\left|Z'\right|\gg\left|Z\right|$ we have $$\left|Z\right|^2 \ll \left|c_1(d_1-d_2)Z'-c_2(d_1-d_2)Z'+c_1(d_3-d_4)Z'\right|$$

\item $\mathbf{R(Z)}$ \textbf{is closed under multiplication but is not closed under addition}. By Lemma \ref{theorem:pivot1} there are then elements $c_1,c_2,c_3,z_4 \in C$ such that for every $Z' \subseteq Z$ with $\left|Z'\right|\gg\left|Z\right|$ we have $$\left|Z\right|^2 \ll \left|(c_1-c_2)Z'+(c_1-c_2)Z'+(c_3-c_4)Z'\right|$$

\item \textbf{$\mathbf{R(Z)}$ is a field}, $G$ say. Lemma \ref{theorem:ratiofield} implies that in this case we have $Z \subseteq aG+b$ for some $a,b \in F$. So, collecting together various facts, we have

\begin{itemize}
\item $Z \subseteq A_*^2-a_1$.
\item $A_*^2$ is a $\left(1,\lambda\right)$-antifield, and therefore so is $A_*^2-a_1$.
\item $Z \subseteq aG+b$ for some $a,b \in F$.
\item $\left|Z\right|\gg \frac{T^{65}}{\left|A\right|^{130}\left|B\right|^{194}}$.
\end{itemize}

So for some $\lambda \approx \frac{T^{65}}{\left|A\right|^{130}\left|B\right|^{194}}$, the definition of a $(2,\lambda)$-antifield implies that $\left|Z\right|\leq \left|G\right|^{\frac{1}{2}}=\left|R(Z)\right|^{\frac{1}{2}}$. Lemma \ref{theorem:pivot2} then implies that there are elements $c_1,c_2,c_3,c_4 \in C$ such that for every $Z' \subseteq Z$ with $\left|Z'\right|\gg\left|Z\right|$ we have
$$\left|Z\right|^2 \ll \left|(c_1-c_2)Z'+(c_3-c_4)Z'\right| $$
\end{enumerate}

\subsubsection{Dealing with Case 1}

Given any  $Z' \subseteq Z$ with $\left|Z'\right|\gg\left|Z\right|$ and any $E \subseteq A_*^2$ with $\left|E\right|\gg \left|A_*^2\right|$, apply Lemma \ref{theorem:plunnecke} with $X=c_1(d_1-d_2)E$ and $k=3$ to get

\begin{align*}
\left|Z\right|^2 &\ll \left|c_1(d_1-d_2)Z'-c_2(d_1-d_2)Z'+c_1(d_3-d_4)Z'\right|\\& \ll \frac{\left|E+Z'\right|\left|c_1E-c_2Z'\right|\left|d_1E-d_2E+d_3Z'-d_4Z'\right|}{\left|A_*^2\right|^2}
\end{align*}

By definition of $\Gamma$ from Claim \ref{theorem:coveringclaim}, there is a subset $S_1 \subseteq A_*^2$ with $\left|S_1\right|\gg \left|A_*^2\right|$ such that $d_1 S_1$ can be covered with $\Gamma$ copies of $A_*^1$. Further, there is a subset $S_2 \subseteq S_1$ with  $\left|S_2\right|\gg\left|S_1\right|\gg \left|A_*^2\right|$ such that $-d_2 S_2$ can be covered with $\Gamma$ copies of $A_*^1$. And there is a subset $S_3 \subseteq S_2$ with $\left|S_3\right|\gg\left|A_*^2\right|$ such that $c_1 S_3$ can be covered with $\Gamma$ copies of $A_*^1$. Set $E=S_3$, so that $d_1E$, $-d_2E$ and $c_1E$ can be covered with $\Gamma$ copies of $A_*^1$ each. 

Similarly, recall that $Z \subseteq A_*^2 - a_1$, and pick $Z' \subseteq Z$ with $\left|Z'\right|\gg\left|Z\right|$ such that $d_3Z'$,$-d_4Z'$ and $-c_2Z'$ can each be covered with $\Gamma$ copies of $A_*^1$ each. Altogether, this means that:  

\begin{align*}
\left|Z\right|^2 &\ll \frac{\Gamma^6 \left|E+Z'\right|\left|A_*^1+A_*^1\right| \left|A_*^1+A_*^1+A_*^1+A_*^1\right|}{\left|A_*^2\right|^2}\\
&\leq \frac{\Gamma^6 \left|A_*^2+A_*^2\right|\left|A_*^1+A_*^1\right| \left|A_*^1+A_*^1+A_*^1+A_*^1\right|}{\left|A_*^2\right|^2}
\end{align*}

Lemma \ref{theorem:plunnecke} and the bound in Claim \ref{theorem:coveringclaim} then give

\begin{align*}
\left|Z\right|^2 &\ll \frac{\Gamma^6 \left|A_*^2+A_*^2\right|\left|A_*^1+A_*^1\right|  \left|A_*^1+c_* A_*^2\right|^4}{ \left|A_*^2\right|^5}\ll\frac{\left|A\right|^{383}\left|B\right|^{573}}{T^{191}}
\end{align*}

Comparing with \eqref{zbound} gives $T \ll \left|A\right|^{\frac{643}{321}}\left|B\right|^{\frac{961}{321}}$, which satisfies the bound in the statement of the proposition.

\subsubsection{Dealing with Case 2}

Given any any $Z' \subseteq Z$ with $\left|Z'\right|\gg \left|Z\right|$ and any $E \subseteq A_*^2$ with $\left|E\right|\gg\left|A_*^2\right|$ we can apply Lemma \ref{theorem:plunnecke} with $X=\left(c_1-c_2\right)E$ and $k=2$ to get

\begin{align*}
\left|Z\right|^2 & \ll \left|(c_1-c_2)Z'+(c_1-c_2)Z'+(c_3-c_4)Z'\right|\\
&\ll \frac{\left|E+Z'+Z'\right|\left|c_1E-c_2E+c_3Z'-c_4Z'\right|}{\left|A_*^2\right|}\\
&\leq \frac{\left|A_*^2+A_*^2+A_*^2\right|\left|c_1 E-c_2E+c_3Z'-c_4Z'\right|}{\left|A_*^2\right|}
\end{align*}

As in Case 1, pick $Z'$ and $E$ so that: 

$$\left|Z\right|^2\ll \frac{\Gamma^4 \left|A_*^2+A_*^2+A_*^2\right|\left|A_*^1+A_*^1+A_*^1+A_*^1\right|}{\left|A_*^2\right|} $$
Lemma \ref{theorem:plunnecke} then gives:
\begin{align*}
\left|Z\right|^2&\ll \frac{\Gamma^4 \left|A_*^1+c_*A_*^2\right|^7}{\left|A_*^1\right|^2 \left|A_*^2\right|^4}\ll\frac{\left|A\right|^{275}\left|B\right|^{411}}{T^{137}}
\end{align*}
\par

Comparing with \eqref{zbound} gives $T \ll \left|A\right|^{\frac{535}{267}}\left|B\right|^{\frac{799}{267}}$,which satisfies the bound in the statement of the proposition.

\subsubsection{Dealing with Case 3}

As with Cases 1 and 2, pick $Z'$ so that

$$\left|Z\right|^2 \ll \left|(c_1-c_2)Z'+(c_3-c_4)Z'\right| \leq \Gamma^4 \left|A_*^1+A_*^1+A_*^1+A_*^1\right|$$

Then Lemma \ref{theorem:plunnecke} gives 
\begin{align*}
\left|Z\right|^2&\ll \frac{\Gamma^4\left|A_*^1+c_*A_*^2\right|^4}{\left|A_*^2\right|^3}\ll\frac{\left|A\right|^{239}\left|B\right|^{357}}{T^{119}}
\end{align*}
Comparing with \eqref{zbound} gives $T \ll \left|A\right|^{\frac{499}{249}}\left|B\right|^{\frac{743}{249}}$, which satisfies the bound in the statement of the proposition.

The proof of the proposition is therefore complete, subject to the proofs of Claims \ref{theorem:bsgclaim}, \ref{theorem:sumsetclaim} and \ref{theorem:coveringclaim}, which are given below.

\subsection{Proof of Claim \ref{theorem:bsgclaim}}

Every point in $P$ is the intersection of a horizontal line in $L_B$ (with $y$-co-ordinate lying in $B$) and a line through the origin in $L_A$ (with gradient lying in $A$). Denote the lines in $L_B$ by $h_b$ for each $b \in B$ and the lines in $L_A$ by $d_a$ for each $a \in A$. Furthermore, for each $b \in B$ define the set $X_b \subseteq F$ by 

$$X_b= \left\{x:(x,b) \in h_b\cap P \right\}$$

Note that $X_b$ is a $(1,\lambda)$-antifield for each $b \in B$ as it is contained in the $(1,\lambda)$-antifield $\left\{x:(x,y) \in P\right\}$

Now, the set of lines $L(P)$ and the set of points $P$ generate $T$ colinear triples. So, by averaging, there are two distinct elements $b_1,b_2 \in B$ such that there are $\frac{T}{|B|^2}$ colinear triples $(p_1,p_2,p_3) \in P \times P \times P$ with $p_1 \in h_{b_1}$ and $p_2 \in h_{b_2}$. 

By Lemma \ref{theorem:popularity} there is then a set $B' \subseteq B$ with $\left|B'\right|\gg \frac{T}{\left|A\right|^2 \left|B\right|^2}$ such that, for each $b \in B'$, there are $\gg \frac{T}{\left|B\right|^3}$ colinear triples $(p_1,p_2,p_3) \in P \times P \times P$ with $p_1 \in h_{b_1}$, $p_2 \in h_{b_2}$ and $p_3 \in h_b$. 

This is the same as saying that for each $b \in B'$ there are $\gg \frac{T}{\left|B\right|^3}$ elements $x_1 \in X_{b_1}$ and $x_2 \in X_{b_2}$ for which $$x_1 \left(1-\frac{b-b_1}{b_2-b_1}\right)+x_2 \left(\frac{b-b_1}{b_2-b_1}\right)\in X_b $$

So for each $b \in B'$, we can apply the Balog-Szemeredi-Gowers theorem (Lemma $\ref{theorem:BSG}$) with $X=\left(1-\frac{b-b_1}{b_2-b_1}\right)X_{b_1}$, $Y=\frac{b-b_1}{b_2-b_1} X_{b_2}$, $n=|A|$, $G=\left\{(x_1,x_2)\in X_{b_1} \times X_{b_2}: x_1 \left(1-\frac{b-b_1}{b_2-b_1}\right)+x_2 \left(\frac{b-b_1}{b_2-b_1}\right)\in X_b \right\}$ and $\alpha=\frac{T}{\left|A\right|^2\left|B\right|^3}$ to find subsets $A_b^1 \subseteq X_{b_1}$  and $A_b^2 \subseteq X_{b_2}$ with

\begin{itemize}
\item $\left|A_b^1 + \left(\frac{b_1-b_2}{b_2-b}-1\right)A_b^2\right|=\left|(1-\frac{b-b_1}{b_2-b_1})A_b^1 + \frac{b-b_1}{b_2-b_1}A_b^2\right|\ll \frac{\left|A\right|^{11} \left|B\right|^{15} }{T^5}$
\item $\left|A_b^1\right|,\left|A_b^2\right| \gg \frac{T}{\left|A\right|\left|B\right|^3}$
\end{itemize}

Moreover, note that $A_b^1$ and $A_b^2$ are both $(1,\lambda)$-antifields for each $b \in B'$ as they are contained in the $(1,\lambda)$-antifields $X_{b_1}$ and $X_{b_2}$ respectively.

By dropping at most one element we may assume that $b_2 \notin B'$. Now let $C'=\left\{\frac{b_1-b_2}{b_2-b}-1:b \in B'\right\}$
and note that the map $b \mapsto \frac{b_1-b_2}{b_2-b}-1$ is a bijection. Define sets $A_c^1, A_c^2$ by $A_c^i=A_{b(c)}^i$ for each $c \in C'$. Then we have 

\begin{itemize}
\item $\left|C'\right|=\left|B'\right|\gg \frac{T}{\left|A\right|^2 \left|B\right|^2}$
\item $\left|A_c^1 + c A_c^2\right| \ll \frac{\left|A\right|^{11} \left|B\right|^{15} }{T^5}$ for each $c \in C'$
\item $\left|A_c^1\right|,\left|A_c^2\right| \gg \frac{T}{\left|A\right|\left|B\right|^3}$ for each $c \in C'$
\end{itemize}

Let $P_c=A_c^1 \times A_c^2$, so that $\left|P_c\right|\gg\frac{T^2}{\left|A\right|^2 \left|B\right|^6}$ for each $c \in C'$. Cauchy-Schwartz implies that:

\begin{align*}
\left|C'\right|\frac{T^2}{\left|A\right|^2 \left|B\right|^6}\ll \sum_{c \in C'}\left|P_c\right|
\leq \left|A\right|\sqrt{\sum_{c,c' \in C'}\left|P_c \cap P_{c'}\right|}
\end{align*}

So there is a particular element $c^*\in C'$ such that
$$\sum_{c \in C'}\left|P_c \cap P_{c^*} \right|\gg \left|C'\right|\frac{T^4}{\left|A\right|^6\left|B\right|^{12}}\gg \frac{T^5}{\left|A\right|^8\left|B\right|^{14}}$$

Lemma \ref{theorem:popularity} then yields a subset $C \subseteq C'$ such that
\begin{itemize}
\item $\left|P_c \cap P_{c^*}\right|\gg \frac{T^4}{\left|A\right|^{6}\left|B\right|^{12}}$ for all $c \in C$
\item $\left|C\right|\gg \frac{T^5}{\left|A\right|^{10}\left|B\right|^{14}}$
\end{itemize}

Note that $\left|P_c \cap P_{c^*}\right|=\left|A_c^1 \cap A_{c^*}^1\right|\left|A_c^2 \cap A_{c^*}^2\right|$ to see that 
$$\left|A_c^1 \cap A_{c^*}^1\right|,\left|A_c^2 \cap A_{c^*}^2\right|\gg \frac{T^4}{\left|A\right|^7\left|B\right|^{12}} $$
for each $c \in C$. This completes the proof of the claim.

\subsection{Proof of Claim \ref{theorem:sumsetclaim}}

The claim is proved by repeated application of Lemma \ref{theorem:plunnecke} and inequalities \eqref{size}, \eqref{sumset} and  \eqref{intersection}:

\subsubsection{Proof of \eqref{claim1}}

Lemma \ref{theorem:plunnecke} and the inequalities \eqref{size} and \eqref{sumset} imply that
\begin{align*}
\left|A_c^1+A_c^1\right|\leq \frac{\left|A_c^1 + c A_c^2  \right|^2}{\left|A_c^2\right|}
\ll \frac{\left|A\right|^{23}\left|B\right|^{33}}{T^{11}}
\end{align*}
Similarly for $\left|A_c^2+A_c^2\right|$, which completes the proof of \eqref{claim1}.

\subsubsection{Proof of \eqref{claim2}}
Lemma \ref{theorem:plunnecke}, and inequalities \eqref{intersection} and \eqref{claim1}, imply that 
\begin{align*}
\left|c_*A_c^2 + cA_c^2\right|&\leq \frac{\left|c_*A_c^2+c_*\left(A_c^2 \cap A_*^2\right)\right|\left|cA_c^2+c_*\left(A_c^2 \cap A_*^2\right)\right|}{\left|A_c^2 \cap A_*^2\right|}\\
&\ll \frac{\left|A_c^2+A_c^2\right|}{\left|A_c^2\cap A_*^2\right|}\left|cA_c^2+c_*\left(A_c^2 \cap A_*^2\right)\right|\\
&\ll \frac{\left|A\right|^{30}\left|B\right|^{45}}{T^{15}}\left|cA_c^2+c_*\left(A_c^2 \cap A_*^2\right)\right|
\end{align*}

Now apply Lemma \ref{theorem:plunnecke} again, with \eqref{sumset} and \eqref{intersection}, to see that 

\begin{align*}
\left|cA_c^2+c_*\left(A_c^2 \cap A_*^2\right)\right|&\ll \frac{\left|\left(A_c^1\cap A_*^1\right)+cA_c^2 \right|\left|c_*\left(A_c^2 \cap A_*^2 \right)+\left(A_c^1 \cap A_*^1 \right) \right|}{\left|A_c^1 \cap A_*^1\right|}\\
&\leq \frac{\left|A_c^1+cA_c^2\right|\left|A_*^1+c_*A_*^2\right|}{\left|A_c^1 \cap A_*^1\right|}\\
& \ll \frac{\left|A\right|^{29}\left|B\right|^{42}}{T^{14}}
\end{align*}
which completes the proof of \eqref{claim2}

\subsubsection{Proof of \eqref{claim3}}
Lemma \ref{theorem:plunnecke}, and inequalities \eqref{sumset}, \eqref{intersection}, \eqref{claim1} and \eqref{claim2}, imply that: 
\begin{align*}
\left|c_*A_*^2+cA_c^2\right|&\leq \frac{\left|c_*A_*^2+c_* \left(A_c^2 \cap A_*^2\right)\right|\left|cA_c^2+c_* \left(A_c^2 \cap A_*^2\right)\right|}{\left|A_c^2 \cap A_*^2\right|}\\
&\leq\frac{\left|A_*^2 + A_*^2\right|\left|c_*A_c^2+cA_c^2\right|}{\left|A_c^2 \cap A_*^2\right|}\\
& \ll \frac{\left|A\right|^{89}\left|B\right|^{132}}{T^{44}}
\end{align*}
which completes the proof of \eqref{claim3}

\subsubsection{Proof of \eqref{claim4}}
Lemma \ref{theorem:plunnecke}, and inequalities \eqref{sumset}, \eqref{intersection}, \eqref{claim1} and \eqref{claim3}, imply that 
\begin{align*}
\left|c_*A_*^2+cA_*^2\right|&\ll \frac{\left|c_*A_*^2+c\left(A_c^2 \cap A_*^2\right)\right|\left|cA_*^2+c\left(A_c^2 \cap A_*^2\right)\right|}{\left|A_c^2 \cap A_*^2 \right|}\\
&\leq \frac{\left|c_*A_*^2 + cA_c^2\right|\left|A_*^2+A_*^2\right|}{\left|A_c^2 \cap A_*^2 \right|}\\
&\ll \frac{\left|A\right|^{119}\left|B\right|^{177}}{T^{59}}
\end{align*}
This completes the proof of \eqref{claim4}, and therefore of the whole claim.

\subsection{Proof of Claim \ref{theorem:coveringclaim}}
Given $D \subseteq A_*^2$, $x \in \mathbb{F}_q$ and $c \in C$, use the covering lemma (Lemma \ref{theorem:covering}) to cover a constant proportion of $cD+x$ with 
$$\frac{\left|cD+\left(A_c^1 \cap A_*^1\right)\right|}{\left|A_c^1 \cap A_*^1\right|}\leq \frac{\left|cA_*^2+\left(A_c^1 \cap A_*^1\right)\right|}{\left|A_c^1 \cap A_*^1\right|}$$
 translates of $A_c^1 \cap A_*^1$, and hence with the same number of translates of $A_*^1$. Lemma \ref{theorem:plunnecke} and the inequalities \eqref{sumset},\eqref{intersection} and \eqref{claim1} then give:
 
\begin{align*}
\frac{\left|cA_*^2+\left(A_c^1 \cap A_*^1\right)\right|}{\left|A_c^1 \cap A_*^1\right|}&\ll \frac{\left|cA_*^2+c\left(A_c^2 \cap A_*^2\right)\right|\left|\left(A_c^1 \cap A_*^1\right)+c\left(A_c^2 \cap A_*^2\right)\right|}{\left|A_c^1 \cap A_*^1\right|\left|A_c^2 \cap A_*^2\right|}\\
&\leq 
\frac{\left|A_*^2+ A_*^2\right|\left|A_c^1 +cA_c^2 \right|}{\left|A_c^1 \cap A_*^1\right|\left|A_c^2 \cap A_*^2\right|}\\
&\ll \frac{\left|A\right|^{48}\left|B\right|^{72}}{T^{24}}
\end{align*}

The proof is similar when $c \in -C$. This completes the proof of the claim.

\section*{Acknowledgements}
The author is grateful to Oliver Roche-Newton and Misha Rudnev for useful discussions and for pointing out various typographical errors in earlier drafts, and to Nick Gill for asking some awkward questions.

\bibliographystyle{plain}
\bibliography{Incidencesbibliography}

\end{document}